\documentclass{amsart}

\usepackage[nobysame,initials]{amsrefs}
\usepackage{latexsym}
\usepackage{hyperref}

\usepackage{xcolor}

\allowdisplaybreaks

\newcommand{\R}{\mathbb{R}}
\newcommand{\E}{\mathbb{E}}
\newcommand{\D}{\mathrm{d}}
\DeclareMathOperator{\vol}{V}

\DeclareMathOperator{\bd}{bd}

\DeclareMathOperator{\PP}{\mathbb{P}}

\newtheorem{theorem}{Theorem}[section]
\newtheorem{lemma}[theorem]{Lemma}

\numberwithin{equation}{section}

\title[Random ball-polytopes in smooth convex bodies]{Random ball-polytopes in smooth convex bodies}

\author{Ferenc Fodor}

\address{Department of Geometry, Bolyai Institute, University of Szeged, Aradi v\'ertan\'uk tere 1, 6720 Szeged, Hungary}

\email{fodorf@math.u-szeged.hu}

\thanks{This research was supported by grant TUDFO/47138-1/2019-ITM of the Ministry for Innovation and Technology, Hungary, and by Hungarian National 	Research, Development and Innovation Office NKFIH grant K 116451.
The author also wishes to thank the MTA Alfr\'ed R\'enyi Institute of Mathematics where part of this work was done while he was a visiting researcher. Thanks are also due to Viktor V\'\i gh for the helpful discussions and comments.}


\begin{document}

\maketitle

\bibliographystyle{amsplain}

\begin{abstract}
We study approximations of smooth convex bodies by random ball-polytopes. We examine the following probability model: let $K\subset\R^d$
be a convex body such that $K$ slides freely in a ball
of radius $R>0$ and has $C^2$ smooth boundary. 
Let $x_1,\ldots, x_n$ be i.i.d. uniform random
points in $K$. For $r\geq R$, let $K^r_{(n)}$ denote the intersection of all radius $r$ closed balls
that contain $x_1,\ldots, x_n$. 
Then $K^r_{(n)}$ is a (uniform) random ball-polytope (of radius $r$) in $K$.
We study the asymptotic properties of the expectation of
the number of facets of $K_{(n)}^r$ as $n\to\infty$. While sufficiently round convex bodies behave in a similar way with respect to random approximation by ball-polytopes as to classical polytopes, an interesting phenomenon can be observed when a unit ball is approximated by unit radius random ball-polytopes: the expected number of facets approaches a finite limit as $n\to\infty$.

\end{abstract}

\section{Introduction and results}
In the theory of random polytopes, one of the oldest and probably most frequently investigated model is when one selects a sample of $n$ independent and identically distributed random points from a convex body chosen according to the uniform probability distribution. The convex hull of these random points is a random polytope contained in the body. A sequence of random polytopes obtained this way tends to the convex body as $n\to\infty$. Since the classical paper of R\'enyi and Sulanke \cite{RS63}, a large part of results concerning this probability model has been in the form of asymptotic formulae (as $n$ tends to infinity) about the behaviour of various geometric quantities of the random polytopes, such as, the number of $i$-dimensional faces, intrinsic volumes, etc. Here we do not give a detailed overview of this extensive topic, instead we refer to the surveys by B\'ar\'any \cite{B2008}, Hug \cite{H13}, Reitzner \cite{R10}, Schneider \cites{Sch04,Sch08, Sch14}, Schneider and Weil \cite{SW08}, Weil and Wieacker \cite{WW1993}.

In this paper we investigate a variant of this much studied model where polytopes are replaced by so-called ball-polytopes, that arise as intersections of congruent closed balls.

Our setting is the $d$-dimensional Euclidean space $\R^d$ with its usual inner product $\langle\cdot,\cdot\rangle$ and norm $|\cdot|$. We denote the origin by $o$. As common, points and vectors in $\R^d$ are not distinguished unless necessary. The Euclidean distance of two points is denoted by $d(\cdot,\cdot)$. The closed unit ball centred at $o$ is $B^d=\{x\in\R^d: |x|\leq 1\}$, the open unit ball is $\mathring{B}^d=\{x\in\R^d: |x|<1\}$, and the unit sphere is $S^{d-1}=\{x\in\R^d: |x|=1\}$. The boundary of a set $X\subset\R^d$ is denoted by $\bd X$. Thus, in particular, $S^{d-1}=\bd B^d$. We use $V(\cdot)$ for the $d$-dimensional Lebesgue measure or volume, and $\sigma_{d-1}(\cdot)$ for the spherical Lebesgue measure on $S^{d-1}$.  It is well-known that $\kappa_d:=V(B^d)=\pi^{d/2}/\Gamma(d/2+1)$, where $\Gamma$ denotes Euler's gamma function, cf. \cite{A64}, and $\omega_d:=\sigma_{d-1}(S^{d-1})=d\kappa_d$, see \cite{Sch14}.

When convenient, we use the Landau symbols to indicate the relation between two functions: If for $f,g:\mathbb{N}\to \R_+$, there exists a constant $\gamma>0$ and a number $n_0\in\mathbb{N}$ such that $f(n)< \gamma g(n)$ for all $n>n_0$, then we write $f\ll g$. If $g\ll f\ll g$, then we denote this by $f\approx g$.

In order to formulate our random model, we start with the following definition. Let $r>0$, and let $x, y\in\R^d$ be two points whose distance is not more than $r$. The intersection $[x,y]_r$ of all radius $r$ closed balls containing $x$ and $y$ is called the {\em $r$-spindle} spanned by $x$ and $y$. A set $C\subset\R^d$ is called {\em $r$-hyperconvex (or $r$-spindle convex)} if $[x,y]_s\subset C$ for all $x,y\in C$. Clearly, an $r$-hyperconvex set is also convex in the classical sense. We call a set $K\subset\R^d$ an {\em $r$-hyperconvex body} if it is compact, $r$-hyperconvex, and has non-empty interior. 

Let $K\subset\R^d$ be an $r$-hyperconvex body for $r>0$. Let $x_1,\ldots, x_n$ be i.i.d. random points from $K$ chosen according to the uniform probability distribution. Let $K_{(n)}^r$ denote the intersection of all radius $r$ closed balls that contain $x_1,\ldots, x_n$. It follows from the $r$-hyperconvex property of $K$ that $K_{(n)}^r\subset K$, see, for example, \cite[Corollary~3.4 on p. 205]{BLNP2007}. We call $K_{(n)}^r$ a uniform random ball-polytope (of radius $r$).


An ordinary polytope is the convex hull of a finite number of points, which is the same as the intersection of all closed half-spaces containing the points. In the construction of $K_{(n)}^r$, the role of closed half-spaces is played by radius $r$ closed balls. This explains the use of the term ball-polytope for $K_{(n)}^r$. The simplest ball-polytope with non-empty interior is the $r$-spindle $[x,y]_r$ of just two points. Thus, unlike in the case of convex polytopes, any set of at least two points in $K$ determines a proper ball-polytope in $\R^d$ (one that has non-empty interior).  
A ball $B$ of radius $r$ is a supporting ball of $K_{(n)}^r$ if $K_{(n)}^r\subset B$ and $K_{(n)}^r\cap \bd B\neq\emptyset$. We call a set $F\subset\bd K_{(n)}^r$ a {\em facet} if $F=K_{(n)}^r\cap \bd B$ for a supporting ball $B$ of radius $r$ and $F$ has positive surface measure on $\bd K$. The example of the spindle shows that a proper ball-polytope may have no facet at all, unlike in the case of an ordinary polytope.  
In this paper we study the number $f_{d-1}(K_{(n)}^r)$ of facets of $K_{(n)}^r$. We do not embark on a detailed investigation of the general facial structure of $K_{(n)}^r$. We note, however, that
if $x_1,\ldots, x_n\in K$ are i.i.d. uniform random points from $K$, and if $K_{(n)}^r$ has any facets, then they are spherical $(d-1)$-simplices with probability one.

Intersections of congruent balls and the associated notion of hyperconvexity have played important roles in the study of several problems recently, such as, for example, the Kneser-Poulsen conjecture, bodies of constant width, diametrically complete bodies, randomized isoperimetric inequalities, etc. For a more complete overview and references see, for example, \cites{BLNP2007, FTGF2015, MMO19, PP17}. Random approximations with intersections of congruent circles were treated, for example, in \cites{FKV14, FV18}. This new probability model can be considered as a generalization of the classical model with random polytopes, however, some of the phenomena that can be observed for balls is different from the classical model. 

We will examine this probability model for $r$-hyperconvex bodies that have sufficiently smooth boundary. 
We say that a convex body $K\subset\R^d$ slides freely in a ball $B$ if for each point $x\in \bd B$, there exist a vector $v\in\R^d$ such that $x\in K+v\subset B$, cf. \cite{Sch14}.
It is known that the following three statements are pairwise equivalent for a convex body $K$: 
i) $K$ is $r$-hyperconvex, 
ii) $K$ is the intersection of all closed balls of radius $r$ containing it, and 
iii) $K$ slides freely in a ball of radius $r$, cf. \cite[Corollary~3.4]{BLNP2007}. 

Let $K\subset\R^d$ be a convex body with $C^2$ smooth boundary. The fact that $K$ slides freely in a ball of radius $r$, or that it is hyperconvex with radius $r$, is equivalent to that all radii of curvature at each boundary point are at most $r$, cf. \cite[Theorem~3.2.12, Corollay~3.2.13]{Sch14}. 

In this paper we are concerned with the expectation $\E f_{d-1}(K^r_{(n)})$ in $C^2$ smooth convex bodies that slide freely in a ball of radius $R$ for some $R>0$. We will see that if $r>R$, then $\E f_{d-1}(K^r_{(n)})$ behaves in a similar manner (in the sense of the order of magnitude in $n$) as uniform random polytopes do in the ordinary $C^2_+$ convex case. 
The major difference occurs when $K$ is a ball of radius $r$. Then $\E f_{d-1}(K^r_{(n)})$ tends to a finite limit as $n\to\infty$. 
 
Fodor, Kevei and V\'\i gh proved (see Theorem~1.3 in \cite{FKV14}) an asymptotic formula for the expected number of vertices of uniform random disc-polygons of radius $r$ in a circular disc of radius $r$ in the plane. It is established in \cite{FKV14} that for $r>0$, it holds that
$$
\lim_{n\to\infty} \E f_0 ((rB^2)_{(n)}^r)=\frac{\pi^2}{2}.
$$

The first main result of this paper is the $d$-dimensional generalization of the above statement. 

\begin{theorem}\label{main-sphere}
For $r>0$ it holds that 
\begin{equation}\label{sphere}
\lim_{n\to\infty} \E f_{d-1}((rB^d)^r_{(n)})=\eta_d\frac{\pi^{d-1}\kappa_d}{\kappa_{d-1}},
\end{equation}
where the constant $\eta_d$ depends only on the dimension.
\end{theorem}
The exact geometric meaning of the constant $\eta_d$ will be explained later. Now we only note that $\eta_2=1$ and $0<\eta_d<1$ for $d\geq 3$. 

Although the proof of Theorem~\ref{main-sphere} is based on similar ideas as that of Theorem~1.3 in \cite{FKV14}, it uses several new tools that are not present in the planar argument in \cite{FKV14} but are necessary because of the general $d$-dimensional setting.

We note that Theorem~\ref{main-sphere} also yields that the limit of the expected number of vertices
of $(rB^d)_{(n)}^r$ is also bounded above by a constant, 
$$\lim_{n\to\infty}\E f_0((rB^d)_{(n)}^r)\leq c(0,d),$$
where $c(0,d)$ depends only on the dimension of the space.

The following asymptotic formula was obtained by Fodor, Kevei and V\'\i gh in \cite{FKV14} as a consequence of Efron's identity \cite{Efron} for the missed area in the planar case
$$\lim_{n\to\infty} \E \vol (rB^2\setminus (rB^2)^r_{(n)})\, n=\frac{r^2\pi^2}{2}.$$

From \eqref{sphere} it follows that
\begin{equation}\label{missedvolsphere}
\E \vol(B^d\setminus B^d_{(n)})\approx \frac 1n
\end{equation}
by the ball-convex version of Efron's identity \cite{Efron}:
\begin{equation}\label{efron}
\E f_{0}(K^r_{(n)})=\frac{n\E \vol(K\setminus K^r_{(n)})}{\vol (K)},
\end{equation}
whose two-dimensional version was proved in \cite[p. 911]{FKV14}. The proof in $d$-dimensions is completely analogous so we omit the detailed argument here.

We note that the fact that the limit in Theorem~\ref{main-sphere} is finite was announced (without proof) in \cite{F17}.
The phenomenon described in Theorem~\ref{main-sphere} has no analogue in the probability model of uniform random polytopes in convex bodies. 

Furthermore, a similar phenomenon was described by B\'ar\'any, Hug, Reitzner and Schneider in \cite{BHRS17} in a different probability model. They proved, cf. \cite[Theorem~3.1]{BHRS17}, that the expected number of facets of random spherical polytopes generated as the spherical convex hull of $n$ i.i.d. uniform random points from a half-sphere tends to a finite limit as $n\to\infty$.  

For the case of general $K$ in the plane it was proved by Fodor, Kevei and V\'\i gh, see part (1.4) of Theorem~1.1 in \cite{FKV14}, that
$$\lim_{n\to\infty} \E f_0(K_{(n)}^r)\, n^{-1/3}=c_K(r),$$
if $\bd K$ is $C^2$ and for the curvature $\kappa$ of $\bd K$ it holds that $\kappa (x)>1/r$ for all $x\in \bd K$.
The quantity $c_K(r)$ is a constant that depeds only on $K$ and $r$. 

The second main result of this paper is the following inequality that establishes the order of magnitude of the exptected number of facets of $K_{(n)}^r$ under the assumption that $\bd K$ is $C^2$ and $K$ slides freely in a ball of radius $R<r$. 

\begin{theorem}\label{main-smooth}
Let $K\subset\R^d$ be a convex body with $C^2$ boundary such that $K$ slides freely in a ball of radius $R>0$. For $r>R$, 
it holds that
\begin{equation}
\E f_{d-1}( K^r_{(n)})\approx n^{\frac{d-1}{d+1}},
\end{equation}
where the implied constants depend only on $K$, $d$ and $r$.
\end{theorem}

Although Theorem~\ref{main-smooth} is not an exact asymptotic formula for the expected number of facets, it extends, at least in some sense, the result of (1.4) of Theorem~1.1 in \cite{FKV14} to $d$-dimensions. Its proof also uses several tools that are new compared to the planar case.

We note that the order of magnitude of $\E f_{d-1} (K^r_{(n)})$ in Theorem~\ref{main-smooth} is the same as for classical uniform random polytopes in the case when $K$ has $C^2_+$ boundary. 
When $K_n$ is the convex hull of $n$ i.i.d. uniform random points from the convex body $K$ with $C^2_+$ smooth boundary, and $f_i(K_n)$ denotes the number of $i$-dimensional faces of $K_n$, then for $0\leq i\leq d-1$, 
\begin{equation}\label{linear-formula}
\lim_{n\to\infty}\E f_{i}(K_n)\, n^{-\frac{d-1}{d+1}}=c_{d,i}\Omega (K),
\end{equation}
where $\Omega (K)$ denotes the affine surface area of $K$, and $c_{d,i}$ is a constant that depends only on the dimension of space. The \eqref{linear-formula} asymptotic formula was proved by R\'enyi and Sulanke \cite{RS63} in the two-dimensional case. B\'ar\'any \cite{B89} established lower and upper bounds of the correct order of magnitude for $\E(f_{d-1}(K_n))$ for general $d$. This exact form of \eqref{linear-formula} is due to Wieacker \cite{Wie} for $i=d-1$, and for general $i$ to Reitzner \cite{R05}. The method of convex floating bodies, that was used by B\'ar\'any \cite{B89}, could naturally be employed in the case of random ball-polytopes as well provided we had the equivalent of the Economic Cap Covering Theorem of B\'ar\'any and Larman \cite{BL88}. Therefore, it would of major interest to prove the ball-convex equivalent of the Economic Cap Covering Theorem. However, it seems to the author that this may require new ideas compared to the Euclidean case.

 

Finally, we also note that so far all results on the random ball-polytope model have been in the plane only, see, for example, \cites{FKV14, FV18}. In this paper we develop some of the techniques that can be used to investigate such problems in $\R^d$ for arbitrary $d$, see Section~\ref{sec:tools}. Beside the already mentioned phenomenon that the expected number of proper facets tends to a finite limit a $n\to\infty$ in a ball, we also point out below the problem of finding the probability that the i.i.d. uniform random points $x_1,\ldots, x_d$ from $K$ are in hyperconvex position. This question, at least for $d$ points, is new even in the context of hyperconvex sets, as it does not occur in the plane, only for $d\geq 3$, and it has no direct analog in classical convexity, for details see below. 

The outline of the rest of the paper is the following: In Section~2, we collect some necessary general tools for our arguments. In Section~3, we study properties of ball-caps of convex bodies. In Section~4, we establish a general formula for the expectation of facet numbers, and Sections~5 and 6 contain the proofs of Theorems~\ref{main-sphere} and \ref{main-smooth}, respectively.

\section{Tools}\label{sec:tools}


We start with the following statement, which is a Blaschke--Petkantschin type transformation formula involving $d$-dimensional spheres of radius $r>0$. It is similar to, for example, \cite[Theorem~7.3.1., p. 287]{SW08}, which was originally proved by Miles \cite{Miles1970}. 

For $v_1,\ldots,v_d\in\R^d$, let $\nabla_d(v_1,\ldots,v_d)$
denote the $d$-dimensional volume of the parallelotope spanned by the vectors $v_1,\ldots, v_d$.
Let $r>0$ be fixed, and consider the  differentiable map 
\begin{equation}\label{T}
T:\R^d\times(S^{d-1})^d\to (\R^d)^d,\quad T(z,u_1,\ldots,u_d)=(z+ru_1,\ldots,z+ru_d).
\end{equation}
Let  $D\subset \R^d\times(S^{d-1})^d$ be a measurable set such that the restriction of $T$ to $D$ is bijective with the possible exception of a set of measure zero. Then the following holds.

\begin{lemma}\label{mainlemma}
If $f:(\R^d)^d\to \R$ is a non negative measurable function, then
\begin{multline}\label{transform}
\int_{T(D)}\cdots\int_{T(D)} f(x_1,\ldots x_d)\; 
\D x_1\ldots \D x_d\\
 =r^{d(d-1)}\int_{\R^d}\int_{S^{d-1}}\cdots\int_{S^{d-1}}
{\mathbf 1}((z,u_1,\ldots, u_d)\in D)\\
 \times f(z+ru_1,\ldots, z+ru_d) \nabla_d(u_1,\ldots,u_d)\;
 \D u_1\ldots\D u_d \D z.
\end{multline}
\end{lemma} 
Here 
$\int_{\R^d}\ldots \D x $ denotes integration with respect to the Lebesgue measure in $\R^d$, and 
$\int_{S^{d-1}}\ldots \D u $ denotes integration with respect to the spherical Lebesgue measure on $S^{d-1}$.
We note that the two-dimensional version of \eqref{transform} was already known to Santal\'o \cite{S46}, and was recently used in \cite{FKV14}, where a short proof of it was also provided for $d=2$. 
\begin{proof}
	We need to show that the Jacobian $|\D T|$ of $T$ is
	$$|\D T|=r^{d(d-1)}\cdot\nabla_d(u_1,\ldots, u_d).$$
	We follow a similar argument and notation as in the proof of Theorem~7.3.1 in Schneider and Weil \cite[pp. 287--288]{SW08} whose idea goes back to M\o ller \cite{Moller1994}.
	
	Vectors of $\R^d$ are considered columns, and $I_d$ is the $d\times d$ identity matrix. For a vector valued differentiable function $v$, the symbol $\dot{v}$ denotes its derivative.  
	
	We assume that in a neighbourhood of the $u_i$ the local coordinate system is chosen such that the $d\times d$ matrix $(u_i\dot{u}_i)$ is orthogonal for all $i$. 
	We recall from \cite[p. 287]{SW08} that for a vector $u\in S^{d-1}$, where the matrix $(u\dot{u})$ is orthogonal, the following hold
	$$\dot{u}^tu=0,\quad \dot{u}^t\dot{u}=I_{d-1},\quad I_d-\dot{u}\dot{u}^t=uu^t.$$
	The Jacobian of $T$ can be written in the following block matrix form
	$$\D T=\left |\begin{matrix}
	I_d & r\dot{u}_1 & 0 & \cdots & 0\\
	I_d & 0 & r\dot{u}_2 & \cdots & 0\\
	\vdots & \vdots & \vdots & \ddots & 0\\
	I_d & 0 & 0 & \cdots & r\dot{u}_d\\
	\end{matrix}\right |.$$
	
	Then it follows that
	\begin{align*}
	r^{-2d(d-1)}(\D T)^2 & =
	\left |
	\begin{matrix}
	I_d & \ldots & I_d\\
	\dot{u}_1^t & \ldots & 0\\
	\vdots & \ddots & \vdots\\
	0 & \ldots & \dot{u}^t_d\\
	\end{matrix}
	\right |
	\cdot \left |
	\begin{matrix}
	I_d & \dot{u}_1 & \ldots & 0\\
	\vdots & \vdots & \ddots & \vdots\\
	I_d & 0 & \ldots & \dot{u}_d
	\end{matrix}
	\right |\\
	& =
	\left |
	\begin{matrix}
	dI_d & \dot{u}_1 & \dot{u}_2 & \ldots & \dot{u}_d\\
	\dot{u}_1^t & I_{d-1} & 0 & \ldots & 0\\
	\dot{u}_2^t & 0 & I_{d-1} & \ldots &\vdots\\
	\vdots & \vdots & \vdots & \ddots & \vdots\\
	\dot{u}_d^t & 0 & 0 & \ldots & I_{d-1}
	\end{matrix}
	\right |\\
	& =\left | dI_d-\sum_{i=1}^{d}\dot{u}_i\dot{u}_i^t \right |
	=\left | \sum_{i=1}^{d} u_i u_i^t \right |\\
	& = 
	\left |
	\left (
	\begin{matrix}
	u_1 & \ldots &u_d
	\end{matrix}
	\right )
	\left (
	\begin{matrix}
	u_1^t\\
	\vdots\\
	u_d^t
	\end{matrix}
	\right )
	\right |
	=
	\left |
	\left (
	\begin{matrix}
	u_1 & \ldots & u_d
	\end{matrix}
	\right )
	\right |^2
	= \nabla_d^2(u_1, \ldots, u_d),
	\end{align*}
	which finishes the proof of the lemma.
\end{proof}



We say that a ball $B$ rolls freely in a convex body $K\subset\R^d$, if for any $x\in \bd K$, there exists a $p\in \R^d$ such that $x\in B+p\subset K$, cf. \cite{Sch14}. 

Let $K\subset\R^d$ be a convex body with $C^2_+$ smooth boundary.  Then there exist positive constants $R\geq \varrho>0$ such that $K$ slides freely in a ball of radius $R$ and a ball of radius $\varrho$ rolls freely in $K$, cf. \cite[Theorem 3.2.12, Corollary 3.2.13]{Sch14}. Let $R_K$ denote the smallest number such that $K$ slides freely in a ball of radius $R_K$.  

Let $\sigma_K:\bd K\to S^{d-1}$ denote the spherical image map which assigns to 
each $x\in\bd K$ the unique outer unit normal $\sigma_K(x)\in S^{d-1}$
to $\bd K$  at $x$. In this particular case, the inverse 
$\sigma_K^{-1}:S^{d-1}\to \bd K$ of the spherical map $\sigma_K$ is also well-defined
and bijective between $S^{d-1}$ and $\bd K$, and to a unit vector $u\in S^{d-1}$
it assigns the unique boundary point $x\in\bd K$ where the outer unit normal
to $\bd K$ is exactly $u$. It is known that both $\sigma_K$ and $\sigma_K^{-1}$
are $C^1$ functions in this particular case, see \cite[pp. 113--115]{Sch14}. 

Let $r\geq R_K$, and define the differentiable map
$\Phi_r : S^{d-1}\times \R_+\to \R^d$ as 
\begin{equation}\label{Phi}
\Phi_r(u,t):=u^{-1}_K(u)-(r+t)u.
\end{equation}

\begin{lemma}\label{main2}
For the Jacobian $|\D \Phi_r|$ it holds that
\begin{equation}
|\D \Phi_r(u,t)|=\left |\sum_{i=0}^{d-1} (-1)^i{d-1\choose i}s_{d-i-1}(u)(r+t)^i\right |, 
\end{equation}
where
$$
s_j(u)={d-1\choose j}^{-1}\sum_{1\leq i_1<\cdots <i_j\leq d-1}r_{i_1}(u)\cdots r_{i_j}(u)
$$
are the normalized elementary symmetric functions of the principal radii of curvature $r_1(u),\ldots, r_{d-1}(u)$ at $\sigma_K^{-1}(u)\in\bd K$.
\end{lemma}
The proof of Lemma~\ref{main2} is quite standard, in fact, using Lemma~\ref{unique-vertex}, it is essentially the same as the one presented on page 122 of Section~2.5 in \cite{Sch14} with the substitution $\lambda=-(r+t)$.

Next, we quote (\cite[(5.6) on page 909]{FKV14}, see also \cite[(11) on page 2290]{BFRV09} and \cite{A64}) the following asymptotic formula.
\begin{lemma}
For any $\beta\geq 0$, $\omega>0$ and $\alpha >0$, it holds that
\begin{equation}\label{betaintegral}
\int_{0}^{g(n)}t^\beta (1-\omega t^\alpha)^n dt\sim \frac{1}{\alpha \omega^{\frac{d+1}{\alpha}}}\cdot \Gamma(\frac{\beta+1}{\alpha})\cdot n^{-\frac{\beta+1}{\alpha}},
\end{equation} 
as $n\to\infty$, assuming that
$$\left (\frac{(\beta+\alpha+1)\ln n}{\alpha \omega n}\right )^{1/\alpha}< g(n)< \omega^{-\frac{1}{\alpha}}$$
for sufficiently large $n$. 
\end{lemma}

We will also use the following result from \cite{BHRS17}. Let $S^{d-1}_+$ denote the closed half-sphere which is above the coordinate hyperplane $x_d=0$. Then
\begin{equation}\label{half-sphere}
\int_{S^{d-1}_+}\cdots \int_{S^{d-1}_+}\nabla (u_1,\ldots, u_d)\; 
\D u_1\ldots \D u_d
=\left (\frac{\omega_{d+1}}{2}\right )^{d-1}.
\end{equation}
We note that \eqref{half-sphere} is a special case of a more general formula, cf. \cite[pages 7--8]{BHRS17}.

\section{Properties of ball caps}
We note that although $r$-hyperconvexity is not an affine invariant notion (balls are generally not preserved under affinities), by using a suitable homothety, one can always achieve that $r=1$. We will use this fact as a normalization in our arguments in order to simplify notation. The results for general $r$ follow by scaling. Of course, such a homothety changes $K$ as well, and thus one cannot assume at the same time that $K$ has unit volume. 

In this section we assume that $\bd K$ is $C^2$ smooth with the extra property that all principal curvatures at each point of $\bd K$ are strictly greater than $1$. Then for the Gaussian curvature it holds that $\kappa(x)>1$ for all $x\in\bd K$. Thus, there exists an $R>1$ such that $K$ slides freely in a ball of radius $R$. Furthermore, in this case, $K$ has the property that for any points $x, x'\in K$, the shorter arc of any unit circle passing through $x$ and $x'$ is contained in $K$, so $K$ is hyperconvex ($1$-hyperconvex).

We will call the intersection of $K$ and the complement of an open unit ball a {\em ball cap}. Ball caps play a similar role in our arguments to usual (linear) caps of convex bodies cut off by hyperplanes. We need to establish some basic facts about ball caps that are in analogy with linear caps, most importantly, that each such cap has a well-defined vertex and height. We note that the two-dimensional case was already treated in \cite{FKV14}. Here we extend the planar statements of \cite{FKV14} (cf. Lemmas~4.1--4.3, pp. 905--906) to $d$-dimensions.

\begin{lemma}\label{unique-vertex}
Let $C=K\setminus (\mathring{B}^d+p)$, $p\in\R^d$ be a non-empty ball cap of a convex body $K$ whose boundary is $C^2$ smooth with all principal curvatures strictly greater than $1$ at every point. Then there exists a unique point $x=x(p)\in C\cap\bd K$ (the {\em vertex}) and a positive real number $t=t(p)>0$ (the {\em height}) such that $p=x-(1+t)\sigma_K(x)$. 
\end{lemma}
\begin{proof}
Let $x_1\in \bd C\cap \bd K$ be a point whose distance from $p$ is maximal. Such a point clearly exists, and it is in the interior of $\bd C\cap \bd K$ with respect to $S^{d-1}$. The hyperplane through $x_1$ and orthogonal to $x_1-p$ is a supporting hyperplane of $K$. Thus, $u_1=(x_1-p)/|x_1-p|\in S^{d-1}$ is the outer unit normal of $K$ at $x_1$ and so $x_1$ is a vertex of $C$. The converse of this statement is also true, that is, if a point $x\in\bd C\cap \bd K$ a vertex of $C$, then $x$ is a point of $C$ whose distance from $p$ is maximal. Note that, due to the assumptions, this maximal distance is larger than $1$. 

We need only check the uniqueness of the vertex. Assume, on the contrary, that there are two vertices, say, $x_1$ and $x_2$ with $x_1\neq x_2$. Since $K$ is hyperconvex, it contains the shorter unit circular arc $\gamma$ connecting $x_1$ and $x_2$ that is in the $2$-plane of $p, x_1$ and $x_2$ and whose centre is on the same side of the line $x_1x_2$ as $p$. However, since $|x_1-p|=|x_2-p|>1$, for any point $x\in\gamma\setminus \{x_1, x_2\}$ it holds that $|x-p|>|x_1-p|$, a contradiction.

\end{proof}

We introduce the following notations. For $u\in S^{d-1}$ and $t\geq 0$, let $C(u,t)$ denote the cap of height $t$ and vertex $x=\sigma_K^{-1}(u)$, and let $V(u,t)=V(C(u,t))$. 

Let us fix $u\in S^{d-1}$ and assume that $x=\sigma_K^{-1}(u)=o$ such that the (unique) supporting
hyperplane of $K$ at $x$ is identified with $\R^{d-1}$ and $u=-e_d$. Since $\bd K$ is $C^2_+$, there exists a convex function $f$ in a sufficiently small open ball around $o$ in $\R^{d-1}$ such that $\bd K$ is the the graph of $f$ in above this neighbourhood. Then 
$$f(z)=\frac{1}{2}Q(z)+o(\|z\|^2) \text{ as } z\to 0$$
in this small neighbourhood. The quadratic form $Q$ is the second fundamental form of $\bd K$ at $x$.
Under the hypotheses of the lemma, $Q$ is positive definite. It is well-known that if we choose a suitable 
orthonormal basis $e_1,\ldots, e_{d-1}$ in $\R^{d-1}$, then
$$Q(z)=k_1z_1^2+\cdots +k_{d-1}z_{d-1}^2$$
for $z=z_1e_1+\ldots +z_{d-1}e_{d-1}\in\R^{d-1}$,  
where the quantities $k_1,\ldots, k_{d-1}$ are the principal curvatures of $\bd K$ at $x=\sigma_K^{-1}(u)$, and the directions determined by the orthonormal basis vectors $e_1,\ldots, e_{d-1}$ are the principal directions. In particular, if $w\in S^{d-2}$, then $k(w)=Q(w)$ is the {\em sectional curvature} of $\bd K$ at $x$ in the direction of $w$. Of course, $k_i=Q(e_i)$ for $i=1,\ldots, d-1$. 

\begin{lemma}
With the same hypotheses as in Lemma~\ref{unique-vertex} and with the notation introduced above, it holds
\begin{align}\label{cap-volume1}
\lim_{t\to 0^+} V(u,t)\cdot t^{-\frac{d+1}{2}}=\frac{2^{\frac{d+1}{2}}\kappa_{d-1}}{d+1}\int_{S^{d-2}}(Q(w)-1)^{-\frac{d-1}{2}} dw
\end{align}
\end{lemma}

\begin{proof}
Now, let $z=\tau w$, where $\tau\geq 0$ and $w\in S^{d-2}$. Assume that $\R^d=\R^{d-1}\times \R$ and that $e_d$ is unit vector such that $e_1,\ldots, e_d$ is an orthonormal basis of $\R^d$ that has a positive orientation, and $e_d=-\sigma_K(x)$, that is, $e_d$ is the inner unit normal of $\bd K$ at $x$. For $t>0$, the sphere $S^{d-1}+(1+t)e_d$ determines a cap of $K$, that has height $t$ and vertex $x$. In a sufficiently small neighbourhood of $o$, the lower hemisphere of $S^{d-1}+(1+t)e_d$ is the graph of the function
$$g_t(\tau w)=\tau w+(1+t-\sqrt{1-\tau^2})e_d.$$  
It is not difficult to check that for a fixed $w\in S^{d-2}$, the intersection point $\tau^*(w)$ of $\bd K$ and the sphere $S^{d-1}+(1+t)e_d$ satisfies 
\begin{equation}\label{tau}
\tau^*(w)=\sqrt{\frac{2}{k(w)-1}}t^{1/2}+o(t^{1/2})\text{ as } \tau\to 0^+.
\end{equation}
Thus,
$$V(u,t)=\int_{S^{d-2}}\int_{0}^{\tau^*(w)}(g_t(\tau w)-f(\tau w))\tau^{d-2}\omega_{d-1}\D \tau \D w.$$
Using Taylor's theorem, we obtain that
\begin{align*}
V(u,t) & =\omega_{d-1}\int_{S^{d-2}}\int_0^{\tau^*(w)}
\left (t+\frac{\tau^2}{2}-\frac{k(w)\tau^2}{2}+o(\tau^2)\right )\tau^{d-2}\D \tau \D w\\
& =\omega_{d-1}\int_{S^{d-2}}\left (\frac{t\tau^{d-1}}{d-1}+\frac{\tau^{d+1}}{2(d+1)}-\frac{k(w)\tau^{d+1}}{2(d+1)}\right |_0^{\tau^*(w)}+o(\tau^{d+1})\, \D w\\
& =\frac{2^{\frac{d+1}{2}}\kappa_{d-1}}{d+1}t^{\frac{d+1}{2}}\int_{S^{d-2}}\left (k(w)-1\right )^{-\frac{d-1}{2}}dw+o\left (t^{\frac{d+1}{2}}\right ) \text{ as } t\to0^+,
\end{align*}
which completes the proof. 
\end{proof}


For a hyperconvex body $K\subset\R^d$, we say that the points $x_1,\ldots, x_d\in K$ are in {\em hyperconvex position} if there exists a $p\in\R^d$ such that $x_{i}\in S^{d-1}+p$ for all $i=1,\ldots, d$. Note that for $d=2$ this is always the case, however, for $d\geq 3$ it is not necessarily so. To see this, one may think of three points in $K$ such that one of them is contained in the interior of the spindle spanned by the other two points. 

Let the points $x_1,\ldots, x_{d}\in K$ be in hyperconvex position. Observe that, unless they are on a great subsphere of a unit sphere, they are on exactly two unit spheres $S^{d-1}+p_-$ and $S^{d-1}+p_+$, and thus determine exactly two distinct (ball-)caps of $K$, namely $C_-(x_1, \ldots, x_{d})=K\setminus (\mathring{B}^d+p_-)$, and $C_+(x_1, \ldots, x_{d})=K\setminus (\mathring{B}^d+p_+)$. If the points are on a great subsphere of a unit sphere, then the two caps coincide.
We may assume without loss of generality that $V(C_-(x_1,\ldots, x_d))\leq V(C_+(x_1,\ldots, x_d))$. For the sake of brevity, henceforth, we will use the following shorthand notations $V_-(x_1,\ldots, x_d)=V(C_-(x_1,\ldots, x_d))$ and $V_+(x_1, \ldots, x_d)=V(C_+(x_1,\ldots, x_d))$.

\begin{lemma}\label{delta-cap}
Let $K\subset\R^d$ be a convex body with $C^2$ boundary and all principal curvatures strictly greater than $1$ at every boundary point. Let $x_1, \ldots, x_d\in K$ be arbitrary points that are in hyperconvex position. With the above hypotheses and notation, there exists a constant $\delta>0$, depending only on $d$ and $K$, such that $V_+(x_1,\ldots, x_d)>\delta$. 
\end{lemma}

\begin{proof} 
Since  $\bd K$ is $C^2$ and all sectional curvatures are strictly larger than $1$ at each $x\in\bd K$, the intersection $K\cap(B^d+p_-)\cap (B^d+p_+)$ can never cover $K$. By compactness, there exists $\delta>0$, depending only on $K$, such that $V(K\setminus ((B^d+p_-)\cap (B^d+p_+)))>2\delta$, 
from which the statement of the lemma follows easily. 
\end{proof}



\section{General statements on facet numbers}
In this section, we derive some general statements about the expectation of facet numbers that will be used in the proofs of Theorems~\ref{main-sphere} and \ref{main-smooth}. As we have already noted before, although the problem is not affine invariant, we may and do assume by a suitable homothety that $r=1$. Let $K\subset\R^d$ be a convex body with $C^2$ smooth boundary that slides freely in the unit ball $B^d$ and hence hyperconvex. Let $V=V(K)$.

Let $X_n=\{x_1,\ldots, x_n\}$ be a sample of $n$ i.i.d. random points from $K$ selected according to the uniform probability distribution. The intersection
$$K_{(n)}:=\bigcap_{{y\in \R^d,}\atop {X_n\subset B^d+y}} (B^d+y)$$
is a random ball-polytope contained in $K$. 


Observe that the probability that any subset of at least $d+1$ of the random points $x_1,\ldots, x_n$ are on a unit sphere $S^{d-1}+p$ is $0$. Furthermore, for any $1\leq i_1<\cdots<i_d\leq n$, the probability that the random points $x_{i_1},\ldots, x_{i_d}$ are on a unique unit sphere $S^k+p$, where $k\leq d-2$ is also $0$. Therefore, if $x_{i_1},\ldots, x_{i_d}$ are on a unit sphere $S^{d-1}+p$, then they are in hyperconvex position with probability $1$, and they span a spherical $(d-1)$-simplex.  

A $d$-tuple $x_{i_1}, \ldots, x_{i_d}$ of points from $X_n$ determines a facet of $K_{(n)}$ if and only if $x_{i_1}, \ldots, x_{i_d}$ are in hyperconvex position and all other points of $X_n$ fall into the complement of one (or both) of the caps $C_+(i_1,\ldots, i_d)$ and $C_-(i_1,\ldots, i_d)$. By the above it is clear that facets of $K_{(n)}$ are spherical $(d-1)$-simplices. Due to the independence of the points in $X_n$, we may assume that $x_{i_1}=x_1, \ldots x_{i_d}=x_d$. 

Let us define the following event
$$A:=\{x_1,\ldots, x_d\text{ are in hyperconvex position}\},$$
and let 
$$\eta_d(K):=\PP(A).$$
Clearly, $\eta_d(K)$ depends only on the dimension $d$ and the convex body $K$. Furthermore, $\gamma_2(K)=1$ for any $K$, and $\eta_d(K)<1$ for all $d\geq 3$ and $K$.

Let $I\subset \{1,\ldots, n\}$ with $|I|=d$, and let $X_I=\{x_i: i\in I\}$. Then, by the independence of $x_1,\ldots x_n$, it holds that
\begin{align}\label{expectation-formula}
\E f_{d-1}(K_{(n)})
&=\sum_{I}\frac{1}{V^n}\int_{K^n}\mathbf{1}(X_I\text{ determines a facet of } K_{(n)})\,\D x_1\ldots \D x_n\notag\\
&={n\choose d}\frac{1}{V^n}\int_{K^n}
{\bf 1}(x_1,\ldots, x_d\text{ determine a facet of } K_{(n)})\,\D x_1\ldots \D x_n\notag\\
&={n\choose d}\frac{1}{V^n}\int_{K^n}
{\bf 1}(x_1,\ldots, x_d\text{ are in hyperconvex position})\notag\\
&\quad\times {\bf 1}(x_{d+1},\ldots, x_n\in C_-(x_1,\ldots, x_d)\text { or } 
C_+(x_1,\ldots, x_d))\,
\D x_1\ldots \D x_n\notag\\
&={n\choose d}\frac{1}{V^d}\int_{K^d}{\bf 1}(A)[\PP(x_{d+1},\ldots, x_n\in C_-(x_1,\ldots, x_d))\notag\\
&\quad+\PP(x_{d+1},\ldots, x_n\in C_-(x_1,\ldots, x_d))]\, \D x_1\ldots \D x_d\notag\\
&={n\choose d}\frac{1}{V^d}\int_{K^d}
\mathbf{1}(A)\left [\left (1-\frac{V_+(x_1,\ldots, x_d)}{V}\right )^{n-d}\right.\notag\\
&\quad+\left.\left (1-\frac{V_-(x_1,\ldots, x_d)}{V}\right )^{n-d}\right ]\, \D x_1\ldots \D x_d,
\end{align}
where $\mathbf{1}(\cdot)$ denotes the indicator function of an event.

We will see that the unit ball $B^d$ is special in the sense that the $\E f_{d-1}(B^d_{(n)})$ approaches a finite
limit as $n$ tends to infinity. A similar phenomenon was pointed out in the paper
by Fodor, Kevei and V\'\i gh, cf.  \cite[Theorem~1.3 (1.7) on p. 902]{FKV14} in the case when $d=2$. 


\section{Proof of Theorem~\ref{main-sphere} --- The case of the unit ball}
It is clear that in this case it enough to evaluate the integral \eqref{expectation-formula} for $K=B^d$, $r=1$, as the statement of the theorem is invariant to scaling.

Consider the cap $C(u,t)$ of $B^d$ with vertex $u\in S^{d-1}$ and height $0<t<2$. In this case
$C(u,t)=B^d\setminus (\mathring{B}^d-tu)$. Let $u=e_d$. Elementary geometry shows that
the distance of the set $S^{d-1}\cap (S^{d-1}-tu)$ (which is  a $(d-2)$-sphere) from
the hyperplane $x_d=0$ is $t/2$. Therefore, the volume $V(u,t)$ of $C(u,t)$ is equal to 
the volume of a centred spherical plank, that is, the volume of the intersection of $B^d$
with a plank (the part of space between two parallel hyperplanes) of width $t$ and symmetric to $o$. Therefore, 
\begin{equation}\label{cap-volume}
\lim_{t\to 0^+}\frac{1}{t}\cdot V(u,t)=\kappa_{d-1}.
\end{equation}
 
From \eqref{expectation-formula}, we obtain
\begin{align}\label{sphere:integral}
\lim_{n\to\infty} \E f_{d-1}(B^d_{(n)})
&=
\lim_{n\to\infty}\frac{1}{\kappa_d^d}{n\choose d}\int_{B^d}\ldots\int_{B^d}\left [\left (1-\frac{V_+(x_1,\ldots,x_d)}{\kappa_d}\right )^{n-d}\right.\\
&\left.\quad+\left (1-\frac{V_-(x_1,\ldots,x_d)}{\kappa_d}\right )^{n-d}\right ]\mathbf{1}(A)\, \D x_1\ldots \D x_d.\notag
\end{align}

If the random points $x_1,\ldots, x_d$ are in hyperconvex position, then we can rewrite \eqref{sphere:integral} with the help of the maps $T$ and $\Phi_1$ (see \eqref{T} and \eqref{Phi}) as follows. Let  $u, u_1,\ldots, u_d\in S^{d-1}$ and $0\leq t\leq 2$ such that
$$T(\Phi_1(u,t),u_1,\ldots, u_d)=(x_1,\ldots, x_d).$$
When $K=B^d$ the Jacobian of the map $\Phi=\Phi_1$ is $|\D \Phi(u,t)|=t^{d-1}$.

Let $S(u,t)=B^d\cap (S^{d-1}-tu)$. Using $\Phi(u,t)$, Lemma~\ref{mainlemma}, and the symmetries of $S^{d-1}$, we get 
\begin{align*}
\lim_{n\to\infty} \E f_{d-1}(B^d_{(n)})&=
\lim_{n\to\infty}\frac{\eta_d}{\kappa_d^d}{n\choose d}\int_{S^{d-1}}\int_0^{2}\int_{S(u,t)}\ldots\int_{S(u,t)}
\left (1-\frac{V(u,t)}{\kappa_d}\right )^{n-d}\\
&\quad\quad\quad\times t^{d-1}
\nabla_d(u_1,\ldots,u_d)\, \D u_1\ldots \D u_d \D t\D u,
\end{align*}
where $\eta_d:=\eta_d(B^d)$.

We now split the domain of integration in the variable $t$. It is sufficient to integrate on the interval $[0, h(n)]$, where $h(n)$ is a sequence defined below. This is a standard technique in such approximation problems. For a similar argument see, for example, \cite[Lemma~5.1 on p. 908]{FKV14}.
Let $h(n)=c\log n/n$, where $c$ is a suitable constant. There exist an $n_0>d$ and $\gamma_1>0$ such that for $n>n_0$ it holds that $h(n)<2$ and for all $h(n)\leq t<2$, $V(u,t)>\gamma_1h(n)$ for any $u\in S^{d-1}$.

Since $\nabla_d(u_1,\ldots,u_d)\leq 1$ for any $u_1,\ldots, u_d\in S^{d-1}$, it follows that
\begin{align*}
&\int_{h(n)}^{2}\int_{(S(u,t))^d}
\left (1-\frac{V(u,t)}{\kappa_d}\right )^{n-d}t^{d-1}
\nabla_d(u_1,\ldots,u_d) \D u_1\ldots \D u_d \D t\\
&\leq2^{d-1}\omega_d^d\int_{h(n)}^{2}
\left (1-\frac{\gamma_1 h(n)}{\kappa_d}\right )^{n-d} \D t\\
&=2^{d-1}\omega_d^d\int_{0}^{2}
\left (1-\frac{\gamma_1 c \log n/n}{\kappa_d}\right )^{n-d} \D t\\
&\leq 2^d\omega_d^d n^{-\frac{c\gamma_1(n-d)}{n\kappa_d}}.	
\end{align*}	
Thus, if $c>d\kappa_d/\gamma_1$, then
\begin{align*}
	\lim_{n\to\infty}&\frac{\eta_d}{\kappa_d^d}{n\choose d}\int_{S^{d-1}}\int_{h(n)}^{2}\int_{(S(u,t))^d}
	\left (1-\frac{V(u,t)}{\kappa_d}\right )^{n-d}t^{d-1}
	\nabla_d(u_1,\ldots,u_d) \D u_1\ldots \D u_d \D t\D u\\
	&\leq\lim_{n\to\infty}\frac{\eta_d}{\kappa_d^d}{n\choose d} 2^d\omega_d^{d+1} n^{-\frac{c\gamma_1(n-d)}{n \kappa_d }} \\
	&=0.
\end{align*}

We define the sequence 
\begin{equation*}
\theta_n(u)=\frac{\eta_d}{\kappa_d^d}{n\choose d}\int_0^{h(n)}\int_{(S(u,t))^d}
\left (1-\frac{V(u,t)}{\kappa_d}\right )^{n-d}t^{d-1}
\nabla_d(u_1,\ldots,u_d) \D u_1\ldots \D u_d \D t.
\end{equation*}
As $\theta_n(u)$ is independent of $u\in S^{d-1}$, we may use the simplified notation $\theta_n(u)=\theta_n$. 
Then
\begin{align*}
\lim_{n\to\infty}\E( f_{d-1}(B^d_{(n)}))&=\lim_{n\to\infty}\int_{S^{d-1}}\theta_n(u)\,du=\omega_d\lim_{n\to\infty}\theta_n.
\end{align*}

Let $\varepsilon\in (0,1)$. From \eqref{half-sphere} and \eqref{cap-volume}, it follows that there exists a $0<t_2<2$ such that 
\begin{equation}
(1-\varepsilon)\left (\frac{\omega_{d+1}}{2}\right )^{d-1}\!\!\!\!<\int_{(S(u,t))^d}
\!\nabla_d(u_1,\ldots,u_d)\, \D u_1\ldots \D u_d<(1+\varepsilon)\left (\frac{\omega_{d+1}}{2}\right )^{d-1},
\end{equation}
\begin{equation}
(1-\varepsilon)t\kappa_{d-1}<V(u,t)<(1+\varepsilon)t\kappa_{d-1}
\end{equation}
for all $t\in (0, t_2)$. Since $\varepsilon$ is arbitrary, we get that
\begin{align*}
\omega_d\lim_{n\to\infty} \theta_n&=
\frac{\eta_d\omega_d}{\kappa_d^d}\left (\frac{\omega_{d+1}}{2}\right )^{d-1}\lim_{n\to\infty}{n\choose d}\int_0^{h(n)}
\left (1-t\frac{\kappa_{d-1}}{\kappa_d}\right )^{n-d}t^{d-1}\D t\\
&=\frac{\eta_d}{\kappa_d^{d-1}}\left (\frac{\omega_{d+1}}{2}\right )^{d-1}\frac{1}{(d-1)!}\lim_{n\to\infty} n^d\int_0^{h(n)}
\left (1-t\frac{\kappa_{d-1}}{\kappa_d}\right )^{n-d}t^{d-1}\D t.
\end{align*}

Now, with $\alpha=1$, $\beta=d-1$ and $\omega=\kappa_d/\kappa_{d-1}$, we obtain from \eqref{betaintegral} that
\begin{align*}
\omega_d\lim_{n\to\infty} \theta_n&=
\frac{\eta_d}{\kappa_d^{d-1}}\left (\frac{\omega_{d+1}}{2}\right )^{d-1}\left (\frac{\kappa_d}{\kappa_{d-1}}\right )^d\frac{1}{(d-1)!}\Gamma(d)\\
&=\eta_d\left (\frac{\omega_{d+1}}{2}\right )^{d-1}\frac{\kappa_d}{\kappa_{d-1}^d}\\
&=\eta_d\frac{\kappa_d\kappa_{d+1}^{d-1}(d+1)^{d-1}}{\kappa_{d-1}^d2^{d-1}}.
\end{align*}
Taking into account that $\kappa_{d+1}/\kappa_{d-1}=2\pi/(d+1)$, we get that
\begin{align*}
\lim_{n\to\infty} \E f_{d-1}(B^d_{(n)})
=\eta_d\frac{\pi^{d-1}\kappa_d}{\kappa_{d-1}},
\end{align*}
which finishes the proof of Theorem~\ref{main-sphere}.

\section{Outline of the proof of Theorem~\ref{main-smooth}}\label{section:general}
Since some of the arguments are similar to those in the proof of Theorem~\ref{main-sphere}, we only give a limited amount of details.

Again, we may assume by a suitable homothety that $r=1$. Let $k>1$ be a fixed number and
$K\subset\R^d$  a convex body with $C^2$ smooth boundary such all sectional curvatures at each boundary point $x\in\bd K$ larger than $k$. In this case $K$ is hyperconvex.

First, we note that it is enough to consider the term of the integral \eqref{expectation-formula} which contains $V_-(x_1,\ldots, x_d)$, that is, the smaller ball-cap, as the contribution of the other term (the larger cap) is negligible. Indeed, for any fixed $\alpha$, it follows from  Lemma~\ref{delta-cap} that 
\begin{align*}
	&\lim_{n\to\infty}n^{\alpha}{n\choose d}\frac{1}{V^d}\int_{K^d}
	\left (1-\frac{V_+(x_1,\ldots, x_d)}{V}\right )^{n-d}\mathbf{1}(A)\, \D x_1\ldots \D x_d\\
	&\quad \leq\lim_{n\to\infty}n^{\alpha}{n\choose d}\frac{\eta_d(K)}{V^d}\int_{K^d} e^{-\frac{\delta}{V}(n-d)} \D x_1\ldots \D x_d\\
	&\quad=0.
\end{align*}
By a similar argument, one can easily verify that it is sufficient to integrate over such $d$-tuples $x_1,\ldots, x_d$ in hyperconvex position for which $V_-(x_1,\ldots,x_d)<\delta$. Thus,
\begin{align*}
	\lim_{n\to\infty}\E f_{d-1}(K_{(n)})\cdot n^{-\frac{d-1}{d+1}}&=
	\lim_{n\to\infty}n^{-\frac{d-1}{d+1}} {n\choose d}\frac{1}{V^d}\int_{K^d}
	\left (1-\frac{V_-(x_1,\ldots, x_d)}{V}\right )^{n-d}\\
	&\quad\quad\times {\bf 1}(V_-(x_1,\ldots,x_d)<\delta)\mathbf{1}(A)\,\D x_1\ldots \D x_d
\end{align*}
Now, we reparametrize the $d$-tuples $x_1,\ldots, x_d$ that are in hyperconvex position
using the function 
$$
T(\Phi (u,t), u_1,\ldots, u_d)=(x_1,\ldots, x_d),
$$
for $u,u_1,\ldots, u_d\in S^{d-1}$ and $t\in \R_+$. 
For $u\in S^{d-1}$ and $t>0$, let 
$$C(u,t)=K\setminus(\mathring{B}^{d}+\Phi(u,t)),$$
$$S(u,t)=K\cap (S^{d-1}+\Phi(u,t)),$$
and $V(u,t)=\vol (C(u,t))$.
Let
\begin{align*}
	\psi(u,t)=\int_{S(u,t)}\ldots\int_{S(u,t)}\nabla_d(u_1,\ldots,u_d)
	\D u_1\ldots \D u_d.
\end{align*}
Further, let
$$s(u,t)=\left |\sum_{i=0}^{d-1} (-1)^i{d-1\choose i}s_{d-i-1}(u)(1+t)^i\right |.$$
By the $C_+^2$ property of $\bd K$, there exists a $0<t_1$ such that $V_-(u,t)\geq \delta$ for all $t_1\leq t$ and $u\in S^{d-1}$.
Using Lemmas~\ref{mainlemma} and \ref{main2}, we have that
\begin{align*}
	\lim_{n\to\infty}&\E f_{d-1}(K_{(n)})\cdot n^{-\frac{d-1}{d+1}}\\
	&=\lim_{n\to\infty}\!n^{-\frac{d-1}{d+1}} {n\choose d}\frac{\eta_d(K)}{V^d}
	\int_{S^{d-1}}\int_{0}^{t_1}
	\!\left (1-\frac{V(u,t)}{V}\right )^{n-d}
	s(u,t)\psi(u,t)
	\D u \D t
\end{align*}
The domain of integration with respect to $t$ can be split further in a standard way similarly to the case of the unit ball. Let $h(n)=(c\ln n/n)^{2/(d+1)}$, for some constant $c>(V/\gamma_2)(d^2+1)/(d(d+1))$. There exist $\gamma_2>0$ and $n_0\in{\mathbb N}$ such that for all $n>n_0$, $h(n)< t_1$, and $V(u,t)>\gamma_2 h(n)^{(d+1)/2}$ for all $h(n)\leq t\leq t_1$ and $u\in S^{d-1}$. Since 
$\nabla_d(u_1,\ldots,u_d)\leq 1$ for all $u_1,\ldots, u_d\in S^{d-1}$, and all the $s_i$ functions are uniformly bounded above, we get that for some suitable constants $\gamma_3, \gamma_4>0$, 
\begin{align*}
	\lim_{n\to\infty}n^{-\frac{d-1}{d+1}} &{n\choose d}\frac{\eta_d(K)}{V^d}
	\int_{S^{d-1}}\int_{h(n)}^{t_1}
	\left (1-\frac{V(u,t)}{V}\right )^{n-d}
	\!\!\!\!\!\!\! s(u,t)\psi(u,t)
	\D u \D t
	\\
	&\leq \gamma_3 \lim_{n\to\infty}n^{-\frac{d-1}{d+1}} {n\choose d}\frac{1}{V^d}
	\int_{S^{d-1}}\int_{h(n)}^{t_1}\left (1-\frac{\gamma_2 c\ln n/n}{V}\right )^{n-d}\D t\D u\\
	&\leq \gamma_4\lim_{n\to\infty}n^{-\frac{d-1}{d+1}} {n\choose d}n^{-\gamma_2c/V}\\
	&= 0.
\end{align*}
Then
\begin{align*}
\lim_{n\to\infty}&\E f_{d-1}(K_{(n)})\, n^{-\frac{d-1}{d+1}}\\
	&=\lim_{n\to\infty}n^{-\frac{d-1}{d+1}}{n\choose d}\frac{\eta_d(K)}{V^d}
\int_{S^{d-1}}\int_{0}^{h(n)}\left (1-\frac{V(u,t)}{V}\right )^{n-d}s(u,t)\psi(u,t)
\D t \D u.
\end{align*}

For $u\in S^{d-1}$ and $n\in\mathbb{N}$, introduce  
\begin{align*}
\theta_n (u)&=n^{-\frac{d-1}{d+1}}{n\choose d}\frac{\eta_d(K)}{V^d}\int_{0}^{h(n)}\left (1-\frac{V(u,t)}{V}\right )^{n-d}s(u,t)\psi(u,t)
\D t
\end{align*}
Therefore
$$\lim_{n\to\infty}\E f_{d-1}(K_{(n)})\, n^{-\frac{d-1}{d+1}}=\lim_{n\to\infty}\int_{S^{d-1}}\theta_n(u)\D u.$$
In order to be able to change the limit and the integral using Lebesgue's  dominated convergence theorem, we must show that the functions $\theta_n(u)$ are uniformly bounded. This follows from  Lemmas~\ref{betaintegral} and \ref{cap-volume1} in a quite standard way using the fact that both $s(u,t)$ and $\psi (u,t)$ are uniformly bounded. For an analogous argument, see, for example, \cite[p. 909]{FKV14}. Thus,
$$\lim_{n\to\infty}\E f_{d-1}(K_{(n)})\, n^{-\frac{d-1}{d+1}}=\int_{S^{d-1}}\lim_{n\to\infty}\theta_n(u)\D u.$$
 
Note that due to the $C^2_+$ smoothness of $\bd K$, $s(u,t)=s(u,0)+O(t)$ as $t\to 0^+$ uniformly for $u\in S^{d-1}$. Thus, for an $\varepsilon>0$, there exists $0<t_\varepsilon<t_2$ such that for all for all $u\in S^{d-1}$ and any $0<t<t_\varepsilon$ it holds that
\begin{equation}\label{ize-eps}
(1-\varepsilon) s(u,0)<s(u,t)<(1+\varepsilon)s(u,0), \text{ and}
\end{equation} 
\begin{equation}\label{vol-eps}
(1-\varepsilon)c(K,u)t^{\frac{d+1}{2}} <V(u,t)<(1+\varepsilon)c(K,u)t^{\frac{d+1}{2}},
\end{equation} 
where $c(K,u)$ is the quantity on the right-hand-side of \eqref{cap-volume1}.

For $x_0,\ldots, x_d\in B^d$, and let $\Delta_d(x_0,\ldots,x_d)$ denote the volume of the simplex with vertices $x_0,\ldots, x_d$. The quantity
$$I(d)=\int_{B^d}\ldots\int_{B^d}\Delta_d(x_0,\ldots,x_d)\D x_0\ldots \D x_d,$$
is known to be a constant depending only on $d$, which fact goes back to Busemann \cite{B53}, see also Miles \cite[(29)]{M71}. It follows by simple scaling that for $r>0$, 
$$r^{d(d+2)}I(d)=\int_{rB^d}\ldots\int_{rB^d}\Delta_d(x_0,\ldots,x_d)\D x_0\ldots \D x_d.$$

If $0<\varrho<1$ is a number such that a ball of radius $\varrho$ rolls freely in $K$, and $0<R<1$ is such that $K$ slides freely in $RB^d$, then let 
$$S_\varrho (u,t)=(S^{d-1}+x-(1+t))\cap (\varrho B^d+x-\varrho u),$$ 
$$S_R (u,t)=(S^{d-1}+x-(1+t))\cap (R B^d+x-R u).$$
Then $$S_\varrho(u,t)\subset S(u,t)\subset S_R(u,t).$$
Let 
$$\psi_\varrho(u,t)=\int_{S_\varrho(u,t)}\ldots\int_{S_\varrho(u,t)} \nabla_d(u_1,\ldots, u_d)\D u_1\ldots \D u_d,$$ 
and let $\psi_R(u,t)$ be defined similarly on $S_R(u,t)$. 
Then
$$\psi_\varrho (u,t)\leq \psi (u,t)\leq \psi_R(u,t).$$
Using \eqref{tau} it is not difficult to see that
$$\lim_{t\to 0^+}\frac{\psi_\varrho(u,t)}{\frac 1d (\frac{2\varrho}{1-\varrho}t)^{\frac{d^2-1}{2}}I(d-1)}=1,\quad  \text{ and }\quad
\lim_{t\to 0^+}\frac{\psi_R(u,t)}{\frac 1d (\frac{{2R}}{1-R}t)^{\frac{d^2-1}{2}}I(d-1)}=1.$$
Therefore, there exist positive constants $\gamma_5$ and $\gamma_6$ such that for any $u\in S^{d-1}$ it holds that for sufficiently small $t$, 
\begin{equation}\label{miles}
 \gamma_5 t^{\frac{d^2-1}{2}}\leq s(u, t)\leq \gamma_6 t^{\frac{d^2-1}{2}}.
\end{equation}
Thus, using \eqref{ize-eps}, \eqref{vol-eps}, \eqref{miles} and Lemma~\ref{betaintegral}, we obtain (disregarding the implied constants) that the following holds uniformly for $u\in S^{d-1}$:
\begin{align*}
\theta_n(u)&\ll \lim_{n\to\infty} n^{-\frac{d-1}{d+1}}{n\choose d}\frac{1}{V^d}\int_{0}^{h(n)}\left (1-\frac{c(K,u)t^{\frac{d+1}{2}}}{V}\right )^{n-d}s(u,0)t^{\frac{d^2-1}{2}}\D t\\
&\ll  n^{-\frac{d-1}{d+1}}{n\choose d}\frac{s(u,0)}{V^d}\int_{0}^{h(n)}\left (1-\frac{c(K,u)t^{\frac{d+1}{2}}}{V}\right )^{n-d}t^{\frac{d^2-1}{2}}\D t\\
&\ll  n^{-\frac{d-1}{d+1}} n^d n^{-\frac{d^2+1}{d+1}}\\
&=1.
\end{align*}
Therefore,
$$\E f_{d-1}(K_{(n)})\ll n^{-\frac{d-1}{d+1}}.$$
The lower bound can be proved similarly. This finishes the proof of Theorem~\ref{main-smooth}.

\end{document}